\documentclass{amsart}  
\usepackage[utf8x]{inputenc} 	   
\usepackage{graphicx}				
\usepackage{amssymb}
\usepackage{amsmath, amsthm}
\usepackage{hyperref}

\newtheorem{thm}{Theorem}[section]
\newtheorem{lem}[thm]{Lemma}

\newtheorem{prop}[thm]{Proposition}
\newtheorem{conj}[thm]{Conjecture}
\newtheorem{que}[thm]{Question}
\theoremstyle{definition} 
\newtheorem{defi}[thm]{Definition}
\newtheorem{exx}[thm]{Example} 
\newenvironment{ex}
{\pushQED{\qed}\exx}
{\popQED\endexx}

\newcommand{\K}{\mathbb K}

\newcommand{\I}{\mathcal I}
\newcommand{\J}{\mathcal J}
\newcommand{\E}{\mathcal E}
\newcommand{\Ll}{\mathcal L}
\newcommand{\Pp}{\mathcal P}
\newcommand{\R}{\mathcal R}

\title{Products of polymatroids with the strong exchange property}
\author{Lisa Nicklasson}
\address{\small Max-Planck Institute for Mathematics in the Sciences, Leipzig, Germany\\
lisa.nicklasson@mis.mpg.de}
\date{}

\begin{document}

\begin{abstract}
	\noindent It was conjectured by White in 1980 that the toric ring associated to a matroid is defined by symmetric exchange relations. This conjecture was extended to discrete polymatroids by Herzog and Hibi, and they prove that the conjecture holds for polymatroids with the so called \emph{strong exchange property}. In this paper we generalize their result to polymatroids that are products of polymatroids with the strong exchange property. This also extends a result by Conca on transversal polymatroids.
\end{abstract}

\maketitle

\section{Polymatroidal bases}
Let $\K$ be a field, and let $\K[X]$ denote the polynomial ring over $\K$ on the set of variables $X=\{x_1, \ldots, x_n\}$. We equip $\K[X]$ with the standard grading, meaning that $\deg(x_1^{\alpha_1} \cdots x_n^{\alpha_n}) = \alpha_1 + \dots + \alpha_n$. We will also use the notation $\deg_i(x_1^{\alpha_1} \cdots x_n^{\alpha_n})=\alpha_i$. 

\begin{defi}
We call a set $B$ of monomials in $\K[X]$ a \emph{polymatroidal basis} if all the monomials in $B$ are of the same degree and the following property holds. If $f$ and $g$ are monomials in $B$ with $\deg_if>\deg_ig$ for some $i$, then there is a $j$ such that $\deg_jf<\deg_jg$ and $\frac{x_j}{x_i}f \in B$. 
\end{defi}
The name \emph{polymatroidal} comes from the fact that a polymatroidal basis is the set of base elements of a discrete polymatroid, phrased in the language of monomials. An ideal generated by a polymatroidal basis is called a \emph{polymatroidal ideal.} The structure of a polymatroidal basis is actually more symmetric than the definition reveals. 

\begin{thm}[Symmetric exchange property]\label{thm:sym_exchange}
Let $B$ be a polymatroidal base, and suppose $f, g \in B$ with $\deg_if>\deg_ig$. Then there is a $j$ with $\deg_jf<\deg_jg$ such that $\frac{x_j}{x_i}f, \frac{x_i}{x_j}g \in B$. 
\end{thm}

For a proof of this theorem, as well as a more extensive background on discrete polymatroids, see \cite{HH}.

To a polymatroidal basis $B$ we can associate the toric ring $\K[B]$. This is the subring of $ \K[X]$ generated by the monomials in $B$. 
In the language of discrete polymatroids $\K[B]$ is also known as the base ring of $B$, or in the language of monomial ideals as the fiber ring of the ideal generated by $B$. If $B=\{f_1, \ldots, f_m\}$, let $Y=\{y_1, \ldots, y_m\}$ and define a homomorphism $\varphi: \K[Y] \to \K[X]$ by $\varphi(y_i)=f_i$. Then $\J_B := \ker \varphi$ is called the \emph{toric ideal of $B$}, and $\K[B]\cong \K[Y]/\J_B$. It is a well known fact that such an ideal is generated by binomials, more precisely by differences of monomials. If we take two variables $y_r, y_s \in Y$, Theorem \ref{thm:sym_exchange} induces binomials of degree two in $\J_B$ in the following way. Say $\varphi(y_r)=f$ and $\varphi(y_s)=g$. Assuming $f \ne g$ Theorem \ref{thm:sym_exchange} tells us that there are $i$ and $j$ such that $\deg_if>\deg_ig$, $\deg_jf<\deg_jg$, and $\frac{x_j}{x_i}f, \frac{x_i}{x_j}g \in B$. Then there are some $y_t, y_u \in Y$ so that $\varphi(y_t) = \frac{x_j}{x_i}f$ and $\varphi(y_u)=\frac{x_i}{x_j}g$ and $y_ry_s-y_ty_u \in \J_B$. The elements in $\J_B$ of this type are called \emph{symmetric exchange relations}. To be precise, $y_ry_s-y_ty_u$ is a symmetric exchange relation if there are $i$ and $j$ such that 
\begin{align*}
&\deg_i \varphi(y_r)> \deg_i \varphi(y_s), \ \ \deg_j \varphi(y_r)< \deg_j \varphi(y_s), \\ 
&\varphi(y_t)=\frac{x_j}{x_i}\varphi(y_r), \ \mbox{ and } \varphi(y_u)=\frac{x_i}{x_j}\varphi(y_s).
\end{align*}

\begin{conj}[White 1980, Herzog-Hibi 2002]\label{conj}
The toric ideal $\J_B$ of a polymatroidal basis $B$ is generated by the symmetric exchange relations. 
\end{conj}

This conjecture was originally posed in \cite{W} for matroids, which correspond to squarefree polymatroidal bases, and for discrete polymatroids in \cite{HH}. Since the symmetric exchange relations are quadratic it is natural to ask whether the algebra $\K[B]$ is Koszul, and if the toric ideal has a quadratic Gröbner basis. Recall that a $\K$-algebra $R$ is Koszul if $\K$ has a linear free resolution over $R$. Every Koszul algebra is defined by an ideal generated in degree two, and every algebra defined by an ideal with a quadratic Gröbner basis is Koszul.

\begin{ex}
The set $B=\{ {x}_{1}{x}_{2}{x}_{3}, {x}_{1}{x}_{3}^{2}, {x}_{2}^{2}{x}_{3}, {x}_{2}{x}_{3}^{2}, {x}_{2}{x}_{3}{x}_{4}, {x}_{3}^{2}{x}_{4}\}$ is a polymatroidal basis. The toric ideal $\J_B$ is the kernel of the homomorphism $\varphi: \K[X] \to \K[Y]$ defined by
\begin{align*}
&\varphi(y_1)={x}_{1}{x}_{2}{x}_{3}, 
\ \ \varphi(y_2)=x_1x_3^2, \ \  \varphi(y_3)=x_2^2x_3, \\
&\varphi(y_4)=x_2x_3^2, \ \ \varphi(y_5)=x_2x_3x_4, \ \  \varphi(y_6)= x_3^2x_4.
\end{align*}
A computation in Macaulay2 \cite{M2} gives that
\[
\J_B=(y_1y_4-y_2y_3, \ y_1y_6-y_2y_5, \ y_3y_6-y_4y_5).
\]
This generating set is a Gröbner basis w.\,r.\,t.\ the Lex order. We also check that the generators are symmetric exchange relations. The fact that $\deg_1(x_1x_3^2)>\deg_1(x_2^2x_3)$, $\deg_2(x_1x_3^2)<\deg_2(x_2^2x_3)$, and
\[
\frac{x_2}{x_1}x_1x_3^3=x_2x_3^2, \ \ \frac{x_1}{x_2}x_2^2x_3=x_1x_2x_3 \ \  \mbox{both belongs to } B
\] 
gives rise to the symmetric exchange relation $y_2y_3-y_1y_4$. Moreover, $\deg_1(x_1x_2x_3)>\deg_1(x_3^2x_4)$, $\deg_4(x_1x_2x_3)<\deg_4(x_3^2x_4)$, and
\[
\frac{x_4}{x_1}x_1x_2x_3=x_2x_3x_4, \ \ \frac{x_1}{x_4}x_3^2x_4=x_1x_3^2 \ \ \mbox{both belongs to } B
\] 
gives the symmetric exchange relation $y_1y_6-y_2y_5$. Last, $y_3y_6-y_4y_5$ is a symmetric exchange relation which comes from the fact that $\deg_2(x_2^2x_3)>\deg_2(x_3^2x_4)$, $\deg_3(x_2^2x_3)<\deg_3(x_3^2x_4)$, and
\[
\frac{x_3}{x_2}x_2^2x_3=x_2x_3^2, \ \ \frac{x_2}{x_3}x_3^2x_4=x_2x_3x_4  \ \ \mbox{both belongs to } B. \qedhere
\]
\end{ex}

We finish this section with a lemma regarding the symmetric exchange relations, which we will need later.

\begin{lem}\label{lem:nonprop_exchange_rel}
Let $B$ be a polymatroidal basis and let $\E_B$ be the ideal generated by the symmetric exchange relations. Suppose we have $y_r, y_s, y_t, y_u$ such that $\varphi(y_t)=\frac{x_j}{x_i}\varphi(y_r)$ and $\varphi(y_u)=\frac{x_i}{x_j}\varphi(y_s)$ for some $i$ and $j$. Then $y_ry_s-y_ty_u \in \E_B$.
\end{lem}
\begin{proof}
If $\deg_i \varphi(y_r)> \deg_i \varphi(y_s)$ and $\deg_j \varphi(y_r)< \deg_j \varphi(y_s)$ then $y_ry_s-y_ty_u$ is a symmetric exchange relation. 

Suppose $\deg_i \varphi(y_r)> \deg_i \varphi(y_s)$ and $\deg_j \varphi(y_r)\ge \deg_j \varphi(y_s)$. Since $B$ is polymatroidal there is some index $k$ such that $ \deg_k \varphi(y_r)< \deg_k \varphi(y_s)$ and $\frac{x_k}{x_i}\varphi(y_r)$, $\frac{x_i}{x_k}\varphi(y_s) \in B$. Then there are some $y_v, y_w$ such that $\varphi(y_v)=\frac{x_k}{x_i}\varphi(y_r)$ and $\varphi(y_w)=\frac{x_i}{x_k}\varphi(y_s)$. It follows that $y_ry_s-y_vy_w\in \E_B$. We also have $y_ty_u-y_vy_w \in \E_B$ because
\[
\deg_j \varphi(y_t)=\deg_j \varphi(y_r)+1> \deg_j \varphi(y_s)-1=\deg_j \varphi(y_u),
\]
\[
\deg_k \varphi(y_t)=\deg_k \varphi(y_r)<\deg_k \varphi(y_s)=\deg_k \varphi(y_u),
\]
\[
\varphi(y_v)=\frac{x_k}{x_i}\varphi(y_r)=\frac{x_k}{x_i}\frac{x_i}{x_j}\varphi(y_t)=\frac{x_k}{x_j}\varphi(y_t) \ \mbox{and similarily} \ \varphi(y_w)=\frac{x_j}{x_k}\varphi(y_u)
\]
Then $y_ry_s-y_ty_u=(y_ry_s-y_vy_w)-(y_ty_u-y_vy_w) \in \E_B$.

Suppose instead $\deg_i \varphi(y_r)\le \deg_i \varphi(y_s)$ and $\deg_j \varphi(y_r)< \deg_j \varphi(y_s)$. Since $B$ is polymatroidal, $\deg_j \varphi(y_s)>\deg_j \varphi(y_r)$ implies that there is some index $k$ such that $\deg_k \varphi(y_s)<\deg_k \varphi(y_r)$ and $\frac{x_k}{x_j}\varphi(y_s), \frac{x_j}{x_k}\varphi(y_s) \in B$. Then there are $y_v$ and $y_w$ such that $\varphi(y_v)=\frac{x_k}{x_j}\varphi(y_s)$ and $\varphi(y_w)=\frac{x_j}{x_k}\varphi(y_r)$, and we have $y_sy_r-y_vy_w \in \E_B$. In a similar way as in the previous case one can verify that $\deg_i \varphi(y_u)> \deg_i \varphi(y_t)$, $\deg_k \varphi(y_u)< \deg_k \varphi(y_t)$, $\varphi(y_v)=\frac{x_k}{x_i}\varphi(y_u)$ and $\varphi(y_w)=\frac{x_i}{x_k}\varphi(y_t)$, so that $y_uy_t-y_vy_w \in \E_B$. 
Then $y_ry_s-y_ty_u=(y_ry_s-y_vy_w)-(y_ty_u-y_vy_w) \in \E_B$.

Last we consider the case $\deg_i \varphi(y_r)\le \deg_i \varphi(y_s)$ and $\deg_j \varphi(y_r) \ge \deg_j \varphi(y_s)$. In this case  the binomial $y_ty_u-y_ry_s$ is a symmetric exchange relation, as $\varphi(y_r)= \frac{x_i}{x_j}\varphi(y_t)$, $\varphi(y_s)=\frac{x_j}{x_i}\varphi(y_u)$, $\deg_i \varphi(y_t)< \deg_i \varphi(y_u)$, and $\deg_j \varphi(y_t) > \deg_j \varphi(y_u)$. We have now seen that $y_ry_s-y_ty_u \in \E_B$ in all possible cases.\end{proof}

\section{The Strong Exchange Property}
It was proved in \cite{HH} that Conjecture \ref{conj} holds for discrete polymatroids with the so called \emph{strong exchange property}.

\begin{defi}
A set $B$ of monomials in $\K[X]$ is called a basis with the \emph{strong exchange property (SEP)} if all the monomials in $B$ are of the same degree and if $f,g \in B$ with $\deg_if>\deg_ig$ and $\deg_jf<\deg_jg$, for some $i$ and $j$, then $\frac{x_j}{x_i}f \in B$. 
\end{defi}
Notice that we will also have $\frac{x_i}{x_j}g \in B$ in the above definition. 

\begin{thm}[Herzog-Hibi 2002]
If $B$ is a basis with the SEP then the toric ideal $\J_B$ is generated by the symmetric exchange relations, and has a quadratic Gröbner basis.  
\end{thm}

We say that a monomial ideal has the SEP if it is generated by a basis with the SEP. Every basis with the SEP is polymatroidal, but the converse does not hold. 
It was proved in \cite{HHV} that bases with the SEP are isomorphic to bases of Veronese type. In Theorem \ref{thm:veronese_type} we make this characterization a bit more explicit. For a basis $B$ with the SEP, let $\max_B(i) = \max\{ \deg_i(f) \ | \ f \in B\}$ and  $\min_B(i) = \min\{ \deg_i(f) \ | \ f \in B\}$.

\begin{thm}\label{thm:veronese_type}
Let $B$ be a basis with the SEP of monomials of degree $d$. Then
\[
B=\left\{ x_1^{\alpha_1} \cdots x_n^{\alpha_n} \ \Big| \ \sum_{i=1}^n\alpha_i=d, \ {\min}_B(i) \le \alpha_i \le {\max}_B(i) \right\}
\]
\end{thm}

\begin{proof}
We begin by proving the following claim: For any $i$ and $j$ there is a monomial $f \in B$ with $\deg_i f=\min_B(i)$ and $\deg_j f=\max_B(j)$. To prove this, let $f$ be a monomial in $B$ such that $\deg_i f= \min_B(i)$ and $\deg_j f$ is as high as possible for a monomial $f$ with $\deg_i f=\min_B(i)$. If $\deg_j f< \max_B(j)$, take $g \in B$ with $\deg_j g=\max_B(j)$. Then $\deg_j g>\deg_j f$ and there is a $k$ with $\deg_k g< \deg_k f$ and $\frac{x_j}{x_k}f \in B$. But this contradicts our choice of $f$, and we can conclude that $\deg_j f=\max_B(j)$. 

A consequence of the claim is that if $f \in B$ and $\deg_i f> \min_B(i)$ and $\deg_j f< \max_B(j)$ we have $\frac{x_j}{x_i}f \in B$. Now let $x_1^{\alpha_1} \cdots x_n^{\alpha_n}$ be a monomial of degree $d$ such that $\min_B(i) \le \alpha_i \le \max_B(i)$ for each $i$. Take any monomial $x_1^{\beta_1} \cdots x_n^{\beta_n} \in B$. If $x_1^{\beta_1} \cdots x_n^{\beta_n} \ne x_1^{\alpha_1} \cdots x_n^{\alpha_n}$ then $\beta_i>\alpha_i$ and $\beta_j<\alpha_j$ for some $i$ and $j$, say $i=1$ and $j=2$. Then $\beta_1>\min_B(1)$ and $\beta_2< \max_B(2)$, and we have $x_1^{\beta_1-1}x_2^{\beta_2+1} x_3^{\beta_3} \cdots x_n^{\beta_n} \in B$. If $x_1^{\beta_1-1}x_2^{\beta_2+1} x_3^{\beta_3} \cdots x_n^{\beta_n} \ne x_1^{\alpha_1} \cdots x_n^{\alpha_n}$ we repeat the same argument. Eventually we will get that $x_1^{\alpha_1} \cdots x_n^{\alpha_n} \in B$, which is what we wanted to prove. 
\end{proof}

We define the product of two polymatroidal bases $B_1$ and $B_2$ as \[B_1B_2=\{b_1b_2 \ | \ b_1 \in B_1, b_2 \in B_2\}.\] If we consider $B_1$ and $B_2$ as generating sets of ideals $I_1$ and $I_2$, then clearly $B_1B_2$ is a generating set for the ideal $I_1I_2$. A nice feature of polymatroidal bases, proved in \cite{CH}, is that their products are again polymatroidal. Products of bases with the SEP does not necessarily have the SEP. For example $\{x_1,x_2\}\{x_3,x_4\}=\{x_1x_3,x_1x_4,x_2x_3,x_2x_4\}$ does not have the SEP. However, powers of bases with the SEP do have the SEP.

\begin{prop}\label{prop:power}
If $B$ has the SEP, so does $B^k$ for any positive integer $k$. 
\end{prop}
\begin{proof}
By Theorem \ref{thm:veronese_type} 
\[
B=\Big\{x_1^{\alpha_1} \cdots x_n^{\alpha_n} \ | \ {\min}_B(i) \le \alpha_i \le {\max}_B(i), \sum_{i=1}^n\alpha_i=d\Big\}.
\]
Let
\[
\mathcal{M}=\Big\{x_1^{\alpha_1} \cdots x_n^{\alpha_n} \ | \ k{\min}_B(i) \le \alpha_i \le k{\max}_B(i), \sum_{i=1}^n\alpha_i=dk\Big\}
\]
We can see that $\mathcal M$ is a basis with the SEP, and $B^k \subseteq \mathcal M$. We shall prove $B^k \supseteq \mathcal M$. Let $x_1^{\alpha_1} \cdots x_n^{\alpha_n} \in \mathcal M$, and write $\alpha_i=k\beta_i+r_i$ with $0 \le r_i <k$. Since $\sum_{i=1}^n \alpha_i=dk$, the number $\sum_{i=1}^nr_i$ must also be divisible by $k$. Since $r_i<k$ we can factorize $x_1^{r_1} \cdots x_n^{r_n}=M_1 \cdots M_k$ where $M_1, \ldots M_k$ are squarefree monomials all of the same degree. Then we have a factorization
\[
x_1^{\alpha_1} \cdots x_n^{\alpha_n}=\prod_{i=1}^k(x_1^{\beta_1} \cdots x_n^{\beta_n}M_i)
\]
and $x_1^{\beta_1} \cdots x_n^{\beta_n}M_i \in B$ since 
\[
{\min}_B(j) \le \left\lfloor \frac{\alpha_j}{k} \right\rfloor \le \deg_j(x_1^{\beta_1} \cdots x_n^{\beta_n}M_i) \le \left\lceil \frac{\alpha_j}{k} \right\rceil \le {\max}_B(j).
\]
We have now proved that $x_1^{\alpha_1} \cdots a_n^{\alpha_n} \in B^k$, and hence $B^k$ has the SEP. 
\end{proof}

The main results of this paper is Theorem \ref{thm:main} which describes the structure of the toric ring $\K[B]$ when $B$ is a product of bases with the SEP. In particular, we show that Conjecture \ref{conj} holds in this case. 

For two graded $\K$-algebras $R=\bigoplus_{i \ge 0}R_i$ and $S = \bigoplus_{i \ge 0}S_i$ we let $R \circ S$ denote their \emph{Segre product}, i.\,e.\ $R \circ S = \bigoplus_{i \ge 0}  R_i \otimes_{\K} S_i$.

\begin{thm}\label{thm:main}
Let $B=B_1 \cdots B_s$ where $B_1, \ldots, B_s$ are bases with the SEP. Then
\begin{enumerate}
\item\label{thm-main1} $\J_B$ is generated by symmetric exchange relations, and
\item\label{thm-main2} $\K[B] \cong \K[B_1] \circ \dots \circ \K[B_s]/L$ where $L$ is generated by binomials of degree one.   
\end{enumerate} 
\end{thm}

The theorem is proved in Section \ref{sec:proof}. Here we should point out that Conjecture \ref{conj} is still open, since not every polymatroidal basis is the product of bases with the SEP. Consider for example the polymatroidal basis \[B=\{x_1x_2,x_1x_3,x_2x_3,x_2x_4,x_3x_4\}.\] It is not of the form in Theorem \ref{thm:veronese_type}, so it does not have the SEP. If $B$ would be a product of two bases, it would be the product of two sets of variables, but one easily checks that this is not the case. 


In the proof of Theorem \ref{thm:main} we will need the following lemma.

\begin{lem}\label{lem:shortcut}
Let $B$ be a basis with the SEP. Suppose $f,g,\frac{x_i}{x_j}f, \frac{x_k}{x_\ell}g \in B$. Then there is an index $m$ such that $\frac{x_i}{x_m}f, \frac{x_\ell}{x_m}f, \frac{x_m}{x_\ell}g \in B$.
\end{lem}
\begin{proof}
If $\frac{x_i}{x_\ell}f \in B$ we are done, as we can choose $m=\ell$. Suppose $\frac{x_i}{x_\ell}f \notin B$. Since $\frac{x_i}{x_j}f \in B$ we have $\deg_if<\max_B(i)$. Then $\frac{x_i}{x_\ell}f \notin B$ implies $\deg_\ell f=\min_B(\ell)$. As $\frac{x_k}{x_\ell}g \in B$ we have $\deg_\ell g>\min(\ell)=\deg_\ell f$. Then there is an $m$ such that $\deg_mg<\deg_mf$ and $\frac{x_m}{x_\ell}g,\frac{x_\ell}{x_m}f \in B$. Also, $\deg_mf>\min_B(m)$ and $\deg_if<\max_B(i)$ gives $\frac{x_i}{x_m}f \in B$.
\end{proof}

\section{Proof of Theorem \ref{thm:main}}\label{sec:proof}
Let $B_1, \ldots, B_s$ be bases in $\K[X]$ with the SEP, and let $B=B_1 \cdots B_s$. We let $Y$ be the set of  $\prod_{j=1}^s|B_j|$ variables $y_v$ indexed by vectors $v=(f_1, \ldots, f_s)$ where $f_j \in B_j$. We define our homomorphism $\varphi: \K[Y] \to \K[X]$ so that $y_v$ maps to the product of the entries in the vector $v$. For a monomial $y_{v_1} \cdots y_{v_d} \in \K[Y]$ we associate a $d\times s$ matrix with row vectors $v_1, \ldots , v_d$. The matrix associated to a monomial is unique up to permutations of the rows, and $\varphi(   y_{v_1} \cdots y_{v_d})$ is the product of all the entries in this matrix. 

Notice that the toric ideal $\J_B= \ker \varphi$, as we defined it here, may contain binomials of degree one, as it can happen that $\varphi(y_v)=\varphi(y_w)$ even though $v \ne w$. Let $\Ll_B$ denote the ideal generated by the binomials of degree one in $\J_B$. Let $\E_B$ be the ideal generated by the symmetric exchange relations in $\J_B$. Also, let $\Pp_B$ be the ideal generated by the binomials $F-G \in \J_B$ such that the matrices of $F$ and $G$ differ in at most one column (up to a permutation of the rows). We will see that $\J_B=\E_B+\Ll_B=\Pp_B+\Ll_B$. This proves that $\J_B$ is generated by the symmetric exchange relations, as the given generating set of $\E_B$ will still correspond to symmetric exchange relations after reducing modulo $\Ll_B$. The fact that $\J_B = \Pp_B + \Ll_B$ proves the second part of Theorem \ref{thm:main}, as 
\[
\frac{\K[Y]}{\Pp_B} \cong \frac{\K[Y_1]}{\J_{B_1}} \circ \dots \circ \frac{\K[Y_s]}{\J_{B_s}}.
 \]
Now it remains to prove $\J_B=\E_B+\Ll_B=\Pp_B+\Ll_B$, or equivalently $\J_B\subseteq (\E_B+\Ll_B) \cap (\Pp_B+\Ll_B)$.

\begin{lem}\label{lem:rearrange}
Let $F$ be a monomial in $\K[Y]$, and let $(f_{ij})$ be the associated matrix. Let $F'$ be the monomial in $\K[Y]$ with the associated matrix $(f'_{ij})$ obtained from $(f_{ij})$ by permuting the elements in one column. Then $F-F' \in \E_B \cap \Pp_B$.
\end{lem}
\begin{proof}
It is clear from the definition of $\Pp_B$ that $F-F' \in \Pp_B$, so we need to prove that $F-F' \in \E_B$. 
As any permutation can by performed as a sequence of transpositions, it is enough to consider a transposition of elements in a column. Without loss of generality, we assume that $f'_{11}=f_{21}$, $f'_{21}=f_{11}$, and otherwise $f'_{ij}=f_{ij}$. We define $u_i=(f_{i1}, \ldots, f_{is})$ and $v_i=(f'_{i1}, \ldots, f'_{is})$, so that $F=y_{u_1} \cdots y_{u_d}$ and $F'=y_{v_1} \cdots y_{v_d}$. Then
\[
F-F'=y_{u_1} \cdots y_{u_d}-y_{v_1} \cdots y_{v_d} = (y_{u_1}y_{u_2}-y_{v_1}y_{v_2})y_{u_3} \cdots y_{u_d},
\] 
and we want to prove that $y_{u_1}y_{u_2}-y_{v_1}y_{v_2} \in \E_B$. Recall that $f_{11},f_{21} \in B_1$, which is a basis with the SEP. We have
\[
f_{21}= \frac{x_{a_1}x_{a_2}}{x_{b_1}x_{b_2}} \cdots \frac{x_{a_k}}{x_{b_k}}f_{11}
\]
for some $a_1, \ldots, a_k, b_1, \ldots b_k$, and we may assume $\{a_1, \ldots, a_k\} \cap \{ b_1, \ldots b_k\} = \emptyset$, so that no variable occurs both in the numerator and in the denominator. By Theorem \ref{thm:veronese_type} 
\[
f_{11}^{(m)}:=\frac{x_{a_1}x_{a_2}}{x_{b_1}x_{b_2}} \cdots \frac{x_{a_m}}{x_{b_m}}f_{11} \in B_1, \ \ m=1, \ldots, k,
\]
since with this construction $\deg f_{11}^{(m)} = \deg f_{11}$ and
\[
{\min}_{B_1}\!(i) \le \min(\deg_i\!f_{11}, \deg_i\!f_{21}) \le \deg_i\!f_{11}^{(m)} \le \max(\deg_i\!f_{11}, \deg_i\!f_{21}) \le {\max}_{B_1}\!(i).
\]
In the same way
\[
f_{21}^{(m)}:=\frac{x_{b_1}x_{b_2}}{x_{a_1}x_{a_2}} \cdots \frac{x_{b_m}}{x_{a_m}}f_{21} \in B_1, \ \ m=1, \ldots, k.
\]
Let
\[
u_1^{(m)} = (f_{11}^{(m)}, f_{12}, \ldots, f_{1s}) \ \mbox{and} \ u_2^{(m)} =(f_{21}^{(m)}, f_{22}, \ldots, f_{2s}), \ \mbox{for} \ m=1, \ldots, k
\]
and notice that $u_1^{(k)} = v_1$ and $u_2^{(k)}=v_2$. By Lemma \ref{lem:nonprop_exchange_rel} 
\[
y_{u_1}y_{u_2} = y_{u_1^{(1)}}y_{u_2^{(1)}}= y_{u_1^{(2)}}y_{u_2^{(2)}} = \ldots =  y_{u_1^{(k)}}y_{u_2^{(k)}}=y_{v_1}y_{v_2} \ \mbox{in} \ \K[Y]/\E_B,
\]
and it follows that $F=F'$ in $\K[Y]/\E_B$. 
\end{proof}

\begin{lem}\label{lem:induction}
Let $F$ and $G$ be monomials in $\K[Y]$ such that their matrices are identical in the first $m<s$ columns, and $F-G \in \J_B$. If $\J_{B'}=\E_{B'}+\Ll_{B'}=\Pp_{B'}+\Ll_{B'}$, where $B'=B_{m+1} \cdots B_s$, then $F-G \in (\E_B+\Ll_B) \cap (\Pp_B + \Ll_B)$. 
\end{lem}
The analogous statement is also true if the matrices of $F$ and $G$ are identical in the last $m$ columns, or more generally in any subset of the columns. The proof in this situation follows the same recipe. 
\begin{proof}
As the first $m$ columns of the matrices $(f_{ij})$ and $(g_{ij})$ representing $F$ and $G$ are identical, and $F-G \in \J_B$, we have 
\[
\prod_{\substack{1 \le i \le d \\ m < j \le s}}f_{ij} = \prod_{\substack{1 \le i \le d \\ m < j \le s}}g_{ij}
\]
which corresponds to a relation in $\J_{B'}$. Then there is a monomial $H \in \K[Y]$ such that $G=H$ in $\K[Y]/((\E_B+\Ll_B) \cap (\Pp_B + \Ll_B))$ and $H$ is identical to $F$ and $G$ in the first $m$ columns, and identical to $F$ in the last $s-m$ columns after possibly permuting the rows. Then we can apply Lemma \ref{lem:rearrange} and get that $G=H=F$ in $\K[Y]/((\E_B+\Ll_B) \cap (\Pp_B + \Ll_B))$.
\end{proof}

\begin{lem}\label{lem:exchange}
Let $F$ be a monomial in $\K[Y]$ such that $\frac{x_1}{x_2}\varphi(F) \in B^d$. Then there are monomials $F'$ and $H$ such that $F-F' \in \E_B\cap \Pp_B$, $\varphi(H)=\frac{x_1}{x_2}\varphi(F')$, and the matrices associated to $F'$ and $H$ differ in at most one entry in each column. 
\end{lem}
\begin{proof}
Let $(f_{ij})$ be the matrix associated to $F$. Since $\frac{x_1}{x_2}\varphi(F) \in B^d$ there are $h_{ij}$'s, such that
\[
\frac{x_1}{x_2}\varphi(F)=\prod_{\substack{1 \le i \le d \\ 1 \le j \le s }}h_{ij}, \ \ \ h_{ij} \in B_j.
\]
For each pair $i,j$ there is a multiset $\Lambda_{ij} \subset \{1, \ldots, n\}^2$ such that 
\[
h_{ij} = \Big(\prod_{(a,b) \in \Lambda_{ij}}\frac{x_a}{x_b}\Big)f_{ij}. 
\]
The choice of the $h_{ij}$'s is not unique, and the goal is the choose them so that for each $j$ there is only one $i$ for which $\Lambda_{ij}$ in nonempty. First note that we can choose $\Lambda_{ij}$ so that for each $(a,b) \in \Lambda_{ij}$ it holds that 
\[\deg_af_{ij}<\deg_ah_{ij} \ \text{and} \ \deg_bf_{ij}>\deg_b,h_{ij}.\]
 It follows that $\frac{x_a}{x_b}f_{ij} \in B_j$, since $B_j$ has the SEP. As 
\[
\Big(\prod_{\substack{1 \le i \le d \\ 1 \le j \le s }} \prod_{(a,b) \in \Lambda_{ij}}\frac{x_a}{x_b}\Big) \varphi(F) =\prod_{\substack{1 \le i \le d \\ 1 \le j \le s }} \prod_{(a,b) \in \Lambda_{ij}}\frac{x_a}{x_b}f_{ij} = \prod_{\substack{1 \le i \le d \\ 1 \le j \le s }}h_{ij} =  \frac{x_1}{x_2}\varphi(F)
\]
we have
\[
\Big(\prod_{\substack{1 \le i \le d \\ 1 \le j \le s }} \prod_{(a,b) \in \Lambda_{ij}}\frac{x_a}{x_b}\Big) = \frac{x_1}{x_2}
\]
and we may order the factors $\frac{x_a}{x_b}$ in the left hand side so that 
\begin{equation}\label{eq:lem_exchange}
\Big(\prod_{\substack{1 \le i \le d \\ 1 \le j \le s }} \prod_{(a,b) \in \Lambda_{ij}}\frac{x_a}{x_b}\Big) = \frac{x_1}{x_{a_1}}\frac{x_{a_1}}{x_{a_2}}\frac{x_{a_2}}{x_{a_3}} \cdots \frac{x_{a_{k-1}}}{x_{a_k}} \frac{x_{a_k}}{x_2}
\end{equation}
Each factor in the right hand side still corresponds to an index pair in some $\Lambda_{ij}$. 
Now, if there are two different values of $i$ such that $\Lambda_{i1}$ is nonempty we do the following. In the right hand side of (\ref{eq:lem_exchange}) consider the leftmost and rightmost factors corresponding to index pairs from any $\Lambda_{i1}$. Say for instance that it is $(a_1,a_2) \in \Lambda_{11}$ and $(a_5,a_6) \in \Lambda_{21}$. Then 
\[
f_{11},f_{21}, \frac{x_{a_1}}{x_{a_2}} f_{11}, \ \text{and} \ \frac{x_{a_5}}{x_{a_6}}f_{21} \ \mbox{all belong to} \ B_1.
\]
By Lemma \ref{lem:shortcut} there is an $m$ such that 
\[
 \frac{x_{a_1}}{x_m} f_{11},\frac{x_{a_6}}{x_m} f_{11}, \ \text{and} \ \frac{x_m}{x_{a_6}}f_{21} \in B_1.
\]
Now we define $F'$ as the matrix $(f'_{ij})$ with $f'_{ij}=f_{ij}$ except $f'_{11}=\frac{x_{a_6}}{x_m} f_{11}$ and $f'_{21}=\frac{x_m}{x_{a_6}}f_{21}$. Notice that $F'-F \in \E_B \cap \Pp_B$. In $\Lambda_{11}$ we replace $(a_1,a_2)$ by $(a_1,a_6)$, and we remove $(a_2,a_3), (a_3,a_4),(a_4,a_5),$ and $(a_5,a_6)$ from their corresponding $\Lambda_{ij}$'s. This means that $\Lambda_{11}=\{(a_1,a_6)\}$ and $\Lambda_{i1}=\emptyset$ for $i>1$. We also redefine the $h_{ij}$'s as
\[
h_{ij} = \Big(\prod_{(a,b) \in \Lambda_{ij}}\frac{x_a}{x_b}\Big)f'_{ij}.
\]
Notice that we still have $h_{ij} \in B_j$, in particular
\[
h_{11}=\frac{x_{a_1}}{x_{a_6}}f'_{11}=\frac{x_{a_1}}{x_{a_6}} \frac{x_{a_6}}{x_m} f_{11}=\frac{x_{a_1}}{x_m}f_{11} \in B_1
\]
and
\[
h_{i1}=f'_{i1} \in B_1 \ \text{for} \ i>1.
\]
Now we do the same procedure for $j=2$, but with $F$ replaced by $F'$ and the right hand side of (\ref{eq:lem_exchange}) replaced by
\[
\frac{x_1}{x_{a_1}}\frac{x_{a_1}}{x_{a_6}}\frac{x_{a_6}}{x_{a_7}} \cdots \frac{x_{a_{k-1}}}{x_{a_k}} \frac{x_{a_k}}{x_2}.
\] 
We continue with this for all values of $j$. Finally we end up with $h_{ij}$'s that for each $j$ there is at most one $i$ for which $h_{ij} \ne f'_{ij}$. Then $H$ given by the matrix $(h_{ij})$ together with $F'$ has the desired properties. 
\end{proof}

Now we are ready to prove that $\J_B\subseteq (\E_B+\Ll_B) \cap (\Pp_B+\Ll_B)$. Suppose we have monomials $F$ and $G$ of degree $d \ge 2$ in $\K[Y]$ such that their difference is in $\J_B$. We shall prove that $F=G$ in $\K[Y]/((\E_B+\Ll_B) \cap (\Pp_B+\Ll_B))$ by induction over $s$. The base case $s=1$ is clear, as then $B=B_1$ has the SEP, and the matrices have only one column. Now suppose the statement is true for any product of fewer than $s$ bases with the SEP. The idea of the proof is the following. Let $(f_{ij})$ and $(g_{ij})$ be the matrices corresponding to $F$ and $G$. Say we are in the special situation where $\prod_{i=1}^df_{i1} = \prod_{i=1}^dg_{i1}.$ Then we also have $\prod_{\substack{1 \le i \le d \\ 2 \le j \le s }}f_{ij}=\prod_{\substack{1 \le i \le d \\ 2 \le j \le s }}g_{ij}$. Let $H$ be the monomial in $\K[Y]$ associated to the matrix 
\[
\begin{pmatrix}
g_{11} & f_{12} & \cdots & f_{1s} \\
g_{21} & f_{22} & \cdots & f_{2s} \\
\vdots & \vdots &        & \vdots \\
g_{d1} & f_{d2} & \cdots & f_{ds} 
\end{pmatrix}
\]
We have $\J_{B'}=(\E_{B'}+\Ll_{B'}) \cap (\Pp_{B'}+\Ll_{B'})$ by the induction hypothesis. By Lemma \ref{lem:induction} we have $H=G$ in $\K[Y]/((\E_B+\Ll_B) \cap (\Pp_B+\Ll_B))$. If $f_{i1}=g_{i1}$ for each $i$ we are done. If not we can apply the induction hypothesis and Lemma \ref{lem:induction} one more time to $F$ and $H$, since they are identical in the last column. It then follows that $F=H=G$ in $\K[Y]/((\E_B+\Ll_B) \cap (\Pp_B+\Ll_B))$. 

Now, what if we are not in the special situation with $\prod_{i=1}^df_{i1} = \prod_{i=1}^dg_{i1}$? If $\prod_{i=1}^df_{i1} \ne \prod_{i=1}^dg_{i1}$ then also $\prod_{\substack{1 \le i \le d \\ 2 \le j \le s }}f_{ij}\ne \prod_{\substack{1 \le i \le d \\ 2 \le j \le s }}g_{ij}$. Since $B_2^d \cdots B_s^d$ is polymatroidal there are say $x_1$ and $x_2$ so that
\begin{equation}\label{eq:proof_mainthm}
\deg_1 \prod_{\substack{1 \le i \le d \\ 2 \le j \le s }}f_{ij} < \deg_1 \prod_{\substack{1 \le i \le d \\ 2 \le j \le s }}g_{ij} \ \text{and} \ \deg_2 \prod_{\substack{1 \le i \le d \\ 2 \le j \le s }}f_{ij}> \deg_2 \prod_{\substack{1 \le i \le d \\ 2 \le j \le s }}g_{ij}
\end{equation}
and 
\[
\frac{x_1}{x_2} \prod_{\substack{1 \le i \le d \\ 2 \le j \le s }}f_{ij} \in B_2^d \cdots B_s^d.
\]
By Lemma \ref{lem:exchange} we can, after replacing $F$ modulo $\E_B\cap \Pp_B$, find $h_{ij}$'s $1 \le i \le d $, $2 \le j \le s$ so that  
\[
\frac{x_1}{x_2} \prod_{\substack{1 \le i \le d \\ 2 \le j \le s }}f_{ij} = \prod_{\substack{1 \le i \le d \\ 2 \le j \le s }}h_{ij}, \ \ \ h_{ij} \in B_j
\]
such that for each $j$ there is at most one $i$'s for which $h_{ij} \ne f_{ij}$. 

Because of (\ref{eq:proof_mainthm}) we also have 
\[
\deg_1\prod_{i=1}^df_{i1}>\deg_1\prod_{i=1}^dg_{i1} \ \text{and} \ \deg_2,\prod_{i=1}^df_{i1}<\deg_2,\prod_{i=1}^dg_{i1}.
\]
By Proposition \ref{prop:power} $B_1^d$ has the SEP, and it follows that 
\[
\frac{x_2}{x_1} \prod_{i=1}^df_{i1} \in B_1^d.
\]
Here we can apply Lemma \ref{lem:exchange} once more, to get that 
\[
\frac{x_2}{x_1} \prod_{i=1}^df_{i1}=\prod_{i=1}^dh_{i1}, \ \ h_{i1} \in B_1
\]
where $h_{i1} \ne f_{i1}$ for at most one $i$. Now, let $H$ be the monomial with the associated matrix $(h_{ij})$, $1 \le i \le d$, $1 \le j \le s$. For each $j$ there is at most one $i$ for which $h_{ij} \ne f_{ij}$. We can use Lemma \ref{lem:rearrange} to, for one $j$ at a time, rearrange the $h_{ij}$'s so that the $i$ where possibly $h_{ij} \ne f_{ij}$ is $i=1$. Since now $F$ and $H$ differ only in the first row they are equal modulo $\Ll_B$. If $\prod_{i=1}^dh_{i1} \ne \prod_{i=1}^dg_{i1}$ we can repeat this process, now with  $H$ and $G$. After finitely many repetitions we will get $\prod_{i=1}^dh_{i1} =\prod_{i=1}^dg_{i1}$. Then we have arrived in the special situation described in the beginning of this proof, and the result follows by induction. 

\section{Rees algebras}
An object closely related to the toric ring $\K[B]$ is the \emph{Rees algebra} of the monomial ideal generated by $B$. In general, the Rees algebra of an ideal $I\subseteq \K[X]$ is defined as $\R(I)= \bigoplus_{j \ge 0} I^jt^j$. If $I=(f_1, \ldots, f_m)$ we have $\R(I)=\K[X][f_1t, \ldots, f_mt]$. Let $Y=\{y_1, \ldots, y_m\}$ and let $\K[X,Y]$ be the polynomial ring on the variables $X \cup Y$ with the bigrading $\deg x_i=(1,0)$ and $\deg y_i = (0,1)$. We say that an element has degree $(k,*)$ if it has degree $(k,s)$ for any $s$, and similarly for $(*,k)$. 
We can present $\R(I)$ as $\K[X,Y]/\ker \phi$ where $\phi:\K[X,Y] \to \K[X][t]$ is the homomorphism defined by $\phi(x_i)=x_i$ and $\phi(y_i)=f_it$. The defining ideal $\ker \phi$ is called the \emph{Rees ideal} of $I$. An ideal $I$ is said to be of \emph{fiber type} if its Rees ideal is generated in degrees $(0,*)$ and $(*,1)$.  

\begin{lem}[{\cite[Lemma 3.2]{N}}]\label{lem:rees}
If $I$ has linear free resolution and is of fiber type then its Rees ideal is generated in degrees $(0,*)$ and $(1,1)$. 
\end{lem} 

Polymatroidal ideals are of fiber type and have linear free resolution, see \cite{HHV} and \cite{HT}, so Lemma \ref{lem:rees} applies to polymatroidal ideals. If $I$ has the SEP then its Rees ideal is generated in degrees $(0,2)$ and $(1,1)$, see \cite[Theorem 3.5]{N}. We can now extend this theorem to products of ideals with the SEP. 

\begin{thm}\label{thm:rees}
The Rees ideal of a product of ideals with the SEP is generated in degrees $(0,2)$ and $(1,1)$.
\end{thm}
\begin{proof}
Let $B$ be a product of bases with the SEP, and let $\I$ be the Rees ideal of the ideal generated by $B$. By Lemma \ref{lem:rees} $\I$ is generated in degrees $(0,*)$ and $(1,1)$. The set of elements of degree $(0,*)$ in $\I$ is precisely the toric ideal $\J_B$. Then the result follows by Theorem \ref{thm:main}.
\end{proof}

As the generators of the Rees ideals in Theorem \ref{thm:rees} has total degree 2 we conclude this section with the following question.

\begin{que}
Is the Rees algebra of a product of ideals with the SEP Koszul?
\end{que}

\section{Gröbner bases and transversal polymatroidal bases}
Polymatroids that are given as products of subsets of $X$ are called \emph{transversal}. This is a special case of the class of polymatroids considered in this paper, as subsets of $X$ indeed have the SEP. Theorem \ref{thm:main} (\ref{thm-main2}) was proved for transversal polymatroids in \cite[Theorem 3.5]{Conca}, which in this case says that $\K[B]$ is isomorphic to the Segre product of polynomial rings modulo an ideal of linear binomials. Moreover \cite[Theorem 3.5]{Conca} states that $\K[B]$ is Koszul when $B$ is transversal. By Theorem \ref{thm:main} (\ref{thm-main1}) we now also know that White's conjecture holds for transversal polymatroids. However, the following questions are still open. 

\begin{que}
Does the toric ideal of a transversal polymatroid, or more generally a product of bases with the SEP, have a quadratic Gröbner basis?
\end{que}
\begin{que}
If $B$ is a product of bases with the SEP, is $\K[B]$ Koszul?
\end{que}

It is not obvious which monomial order to use when searching for a quadratic Gröbner basis of $\J_B$. Let $B=X_1 \cdots X_s$, where $X_1, \ldots, X_s$ are subsets of $X$. Then $\J_B$ is generated by the so called \emph{Hibi relations} together with binomials of degree one, see \cite[Proposition 3.7]{Conca}. The Hibi relations are defined in the following way. We fix an order of the variables in each $X_i$. The orderings do not have to be compatible with each other. Then we let $Y$ be the set of variables indexed by integer vectors $\alpha=(\alpha_1, \ldots, \alpha_s)$ so that $1 \le \alpha_j \le |X_j|$. The toric ideal $\J_B$ is the kernel of $\varphi: K[Y] \to \K[X]$ where $\varphi(y_\alpha)$ is defined as the product of the $\alpha_1$-th element of $X_1$, the $\alpha_2$-th element of $X_2$, and so on. The Hibi relations are all relations $y_\alpha y_\beta - y_{\alpha \wedge \beta}y_{\alpha \vee \beta}$ where $\alpha \wedge \beta = (\max(\alpha_1, \beta_1) , \ldots, \max(\alpha_s, \beta_s))$ and $\alpha \vee \beta =  (\min(\alpha_1, \beta_1) , \ldots, \min(\alpha_s, \beta_s)).$
We have a partial ordering on the variables $Y$ by $y_\alpha \le y_\beta$ if $\alpha_i \le  \beta_i$ for all $i$. Notice that the Hibi relation $y_\alpha y_\beta - y_{\alpha \wedge \beta}y_{\alpha \vee \beta}$ is nonzero only if $y_\alpha$ and $y_\beta$ are incomparable. The Hibi relations are a Gröbner basis w.\,r.\,t.\ the DegRevLex order with any order on $Y$ extending the partial order. The leading term of $y_\alpha y_\beta - y_{\alpha \wedge \beta}y_{\alpha \vee \beta}$ is $y_\alpha y_\beta$. We make the following observation.

\begin{prop}\label{prop:trans_GB}
Let $B=X_1 \cdots X_s$ be a transversal polymatroidal basis such that $\J_B$ is generated by the Hibi relations together with one binomial of degree one. Then $\J_B$ has a quadratic Gröbner basis. Moreover, if $|X_1|= \dots = |X_s|$ then $\K[B]$ is Gorenstein. 
\end{prop}
\begin{proof}
We may order the variables in each $X_i$ so that the one binomial of degree one in $\J_B$ is $y_\alpha-y_\beta$ with $\alpha=(|X_1|, \ldots, |X_s|)$ and $\beta=(1,1, \ldots, 1)$. The Hibi relations are a Gröbner basis, and we shall prove that they remain so after replacing every occurrence of $y_\alpha$ by $y_\beta$. First notice that $y_\alpha$ does not divide the leading term of any Hibi relation. Indeed, $y_\alpha \ge y_\gamma \ge y_\beta$ for any $y_\gamma \in Y$. This shows that replacing $y_\alpha$ by $y_\beta$ does not change the leading term of any binomial in $\J_B$.

Now assume $|X_1|= \dots = |X_s|=m$. Then $\K[B] \cong S/L$ there $S$ is the Segre product of $s$ copies of the polynomial ring in $m$ variables, and $L$ is generated by a binomial of degree one. The Segre product of several of copies of a Gorenstein ring is again Gorenstein, by a result in \cite{GW}. The polynomial ring is Gorenstein, so $T$ is Gorenstein. $T$ is also a domain, and hence $T/L$ is Gorenstein. 
\end{proof}

\begin{ex}
Let $B=\{x_1,x_2\}\{x_2,x_3\}\{x_3,x_4\}\{x_4,x_5\}\{x_5,x_1\}$. Then $Y$ is the set of variables indexed by integer vectors $\alpha=(\alpha_1, \ldots, \alpha_5)$ where $1 \le \alpha_i \le 2$. We have $\varphi(y_{11111})=x_1x_2x_3x_4x_5$ and $\varphi(y_{22222})=x_2x_3x_4x_5x_1$, and $y_{22222}-y_{11111}$ is the only linear relation in $\J_B$. Hence the Hibi relations, with every occurrence of $y_{22222}$ replaced by $y_{11111}$, is a quadratic Gröbner basis defining $\K[B]$. One example of a Hibi relation is $y_{11122}y_{12221}-y_{11121}y_{12222}$ where
\begin{align*}
&\varphi(y_{11122})=x_1x_2x_3x_5x_1, \ \  \varphi(y_{12221})=x_1x_3x_4x_5x_5,\\
& \varphi(y_{11121})=x_1x_2x_3x_5x_5, \ \mbox{and} \ \varphi(y_{12222})=x_1x_3x_4x_5x_1.
\end{align*}
The Segre product of five polynomial rings in two variables is a Gorenstein algebra with Hilbert series 
\[
\frac{1+26t+66t^2+26t^3+t^4}{(1-t)^6}.
\]
Hence $K[B]$ is Gorenstein with Hilbert series 
\[
\frac{1+26t+66t^2+26t^3+t^4}{(1-t)^5}. \qedhere
\]
\end{ex}
 We finish with a question inspired by Proposition \ref{prop:trans_GB}. 
\begin{que}
Which transversal polymatroidal bases define Gorenstein algebras?
\end{que}

%

\section*{Acknowledgement}
I would like to thank Aldo Conca for our discussion about Gröbner bases and transversal polymatroids which led to the content of Section 5. 

\bibliography{polymatroidal}{}
\bibliographystyle{plain}

\end{document}